\documentclass[12pt,english]{amsart}
\usepackage[T1]{fontenc}
\usepackage[latin9]{luainputenc}
\usepackage{geometry}
\usepackage{babel}
\usepackage{prettyref}
\usepackage{amstext}
\usepackage{amsthm}
\usepackage{amssymb}
\usepackage{graphicx}
\usepackage[unicode=true,pdfusetitle,
 bookmarks=true,bookmarksnumbered=false,bookmarksopen=false,
 breaklinks=false,pdfborder={0 0 1},backref=false,colorlinks=false]
 {hyperref}

\makeatletter
\numberwithin{equation}{section}
\numberwithin{figure}{section}
\theoremstyle{plain}
\newtheorem{thm}{\protect\theoremname}[section]
\theoremstyle{definition}
\newtheorem{defn}[thm]{\protect\definitionname}
\theoremstyle{plain}
\newtheorem*{prop*}{\protect\propositionname}
\theoremstyle{definition}
\newtheorem{example}[thm]{\protect\examplename}
\theoremstyle{plain}
\newtheorem{lem}[thm]{\protect\lemmaname}
\theoremstyle{remark}
\newtheorem{rem}[thm]{\protect\remarkname}
\theoremstyle{plain}
\newtheorem{prop}[thm]{\protect\propositionname}
\theoremstyle{plain}
\newtheorem{cor}[thm]{\protect\corollaryname}
\theoremstyle{definition}
\newtheorem*{defn*}{\protect\definitionname}
\theoremstyle{plain}
\newtheorem*{thm*}{\protect\theoremname}
\theoremstyle{remark}
\newtheorem*{claim*}{\protect\claimname}

\makeatother

\providecommand{\claimname}{Claim}
\providecommand{\corollaryname}{Corollary}
\providecommand{\definitionname}{Definition}
\providecommand{\examplename}{Example}
\providecommand{\lemmaname}{Lemma}
\providecommand{\propositionname}{Proposition}
\providecommand{\remarkname}{Remark}
\providecommand{\theoremname}{Theorem}

\begin{document}
\title{Three Classification Results in The Theory of Weighted Hardy Spaces
On The Ball}
\author{Danny Ofek, Gilad Sofer}
\begin{abstract}
We present a natural family of Hilbert function spaces on the $d$-dimensional
complex unit ball and classify which of them satisfy that subsets
of the ball yield isometrically isomorphic subspaces if and only if
there is an analytic automorphism of the ball taking one to the other.
We also characterize pairs of weighted Hardy spaces on the unit disk
which are isomorphic via a composition operator by a simple criterion
on their respective sequences of weights.
\end{abstract}

\maketitle

\global\long\def\bbR{\mathbb{R}}%

\global\long\def\ccR{\mathcal{R}}%

\global\long\def\dlim{\operatorname{\underrightarrow{{\rm lim}}}}%

\global\long\def\Ker{\operatorname{\rm Ker}}%

\global\long\def\End{\operatorname{\rm End}}%

\global\long\def\Aut{\operatorname{Aut}}%

\global\long\def\span{\operatorname{span}}%

\global\long\def\Rank{\operatorname{Rank}}%
 
\global\long\def\Log{\operatorname{\rm Log}}%
 
\global\long\def\Arg{\operatorname{\rm Arg}}%

\global\long\def\myint#1#2#3{\int_{#1}^{#2}\sin#3dx}%

\section{Introduction}

The problem of classification of complete Pick Hilbert function spaces
and their multiplier algebras has been considered by several authors
in the past, see \cite{key-222,key-22,key-6,key-97,M12,M17,key-7,key-3-1}
and also the survey \cite{key-10}. In this paper we go beyond complete
Pick spaces, obtaining some new classification results for weighted
Hardy spaces on the disc and on the ball. We also recover some previous
results with simplified proofs, for spaces as well as for multiplier
algebras.\footnote{\emph{Keywords and Phrases:} Hilbert function spaces, weighted Hardy
space, Drury-Arveson space, multiplier algebras.

\emph{2010 Mathematics Subject Classification Class}: 46E22.

Danny Ofek is an undergraduate student at Tel-Aviv University. Gilad
Sofer is a graduate student at the Technion Institute of Technology.
For any questions or suggestions regarding this paper please contact
us at Dannyofek@mail.tau.ac.il or at Gilad.sofer@campus.technion.ac.il.}

Recall that a \emph{Hilbert function space} on a set $X$ is a Hilbert
space $\mathcal{H}\subset\mathbb{C}^{X}$ for which the evaluation
functionals $eval_{x}:f\mapsto f(x)$ are bounded. These are also
known as\emph{ reproducing kernel Hilbert spaces}, or RKHS. By the
Riesz representation theorem, if $\mathcal{H}$ is a Hilbert function
space on $X$, then for any $x\in X$ there exists $k_{x}^{\mathcal{H}}\in\mathcal{H}$
such that for every $f\in\mathcal{H}$:
\[
\langle f,k_{x}^{\mathcal{H}}\rangle=f(x).
\]

The kernel $k^{\mathcal{H}}$ of $\mathcal{H}$ is the function obtained
by un-currying $x\mapsto k_{x}^{\mathcal{H}}$ , explicitly:
\[
k^{\mathcal{H}}(x,y)=k_{y}^{\mathcal{H}}(x)=\langle k_{y}^{\mathcal{H}},k_{x}^{\mathcal{H}}\rangle.
\]

The first result we will present is a partial classification of weighted
Hardy spaces on the unit disk, which we now define. We denote by $\mathbb{C}[[z]]$
the algebra of formal power series in $z$.

\begin{defn}
Let $w=(w_{n})_{n=0}^{\infty}$ be a sequence of positive numbers,
we define the \emph{weighted Hardy space }corresponding to $w$ as
the vector space:
\[
\mathcal{H}_{w}=\left\{ \sum a_{n}z^{n}\in\mathbb{C}[[z]]:\sum|a_{n}|^{2}w_{n}<\infty\right\} .
\]

With inner product defined by the formula:
\[
\langle f,g\rangle=\sum f_{n}\overline{g_{n}}w_{n}.
\]

As the reader may easily verify, $\mathcal{H}_{w}$ is always a Hilbert
space. The following proposition gives a sufficent and necessary condition
for it to be a Hilbert function space.
\end{defn}

\begin{prop*}
(see \cite[Exercise 2.1.10]{key-4}) Let $w=(w_{n})_{n=0}^{\infty}$
be a sequence of positive numbers, then $\mathcal{H}_{w}$ is a Hilbert
function space on $\mathbb{D}$ if and only if the power series $\sum w_{n}^{-1}z^{n}$
has radius of convergence at least 1. If $\mathcal{H}_{w}$ is a Hilbert
function space on $\mathbb{D}$ then its reproducing kernel at $x$
is given by:
\[
k_{x}^{\mathcal{H}_{w}}(z)=\sum w_{n}^{-1}(\overline{x}z)^{n}.
\]
\end{prop*}
Notice that if the elements of $\mathcal{H}_{w}$ converge to functions
on the disk, then they are automatically holomorphic because they
are defined by power series.\\

We recall the definition of a morphism of Hilbert function spaces. 

A bounded operator $T:\mathcal{H}\to\mathcal{E}$ is a \emph{morphism
of Hilbert function spaces}, if $\mathcal{H},\mathcal{E}$ are Hilbert
function spaces on sets $X,Y$ repsectively and there exist $\phi:X\to Y$,
$f:X\rightarrow\mathbb{C}$ such that for any $x\in X:$
\[
T(k_{x}^{\mathcal{H}})=f(x)k_{\phi(x)}^{\mathcal{E}}.
\]

A morphism which is an isometry of the underlying Hilbert spaces is
called an \emph{isometry of Hilbert function spaces}. Throughout this
paper all morphisms are morphisms of Hilbert function spaces. When
we say two Hilbert function spaces are isomorphic or isometric, we
mean that they are so as Hilbert function spaces.

These definitions suffice to state the first question we tackled:

\subsection{When are two weighted Hardy spaces isomorphic?\protect \\
}

We wish to determine when two weighted Hardy spaces are (isometrically)
isomorphic via an RKHS isomorphism $T$. If $\mathcal{H}_{w}\cong\mathcal{H}_{u}$,
then there is some bijective and bounded linear map $T:\mathcal{H}_{w}\rightarrow\mathcal{H}_{u}$,
a bijection $\phi:\mathbb{D}\rightarrow\mathbb{D}$ and a non-vanishing
function $\lambda:\mathbb{D}\rightarrow\mathbb{C}$ such that:
\[
\forall s\in\mathbb{D}:T(k_{s}^{\mathcal{H}_{w}})=\lambda(s)k_{\phi(s)}^{\mathcal{H}_{u}}.
\]

It turns out that such RKHS isomorphisms can be understood more simply
through their adjoints, which obtain the simple form of weighted composition
operators:
\[
T^{*}h=M_{f}C_{\phi}h=f\cdot\left(h\circ\phi\right).
\]
Here, $\cdot$ can be understood as the pointwise multiplication of
functions, and $f\left(s\right)=\overline{\lambda\left(s\right)}$. 

Our main result gives a sufficient condition for when $\mathcal{H}_{w}\cong\mathcal{H}_{u}$,
and also a necessary condition under the further assumption that $T^{*}$
is a scalar multiple of a composition operator (i.e - $f=const$).
For the statement of the theorem, we make the following definition;
Given two positive sequences, $a_{n}$ and $b_{n}$, we say that $a_{n}\sim b_{n}$
if there exist $\epsilon>0,\,M>0$ such that $0<\epsilon<\frac{a_{n}}{b_{n}}<M$.
\begin{thm}
\label{thm:1.2}If $w_{n}\sim u_{n}$, then $\mathcal{H}_{w}\cong\mathcal{H}_{u}$.
The isomorphism can be chosen to be isometric if and only if there
exists $c>0$ such that $\frac{w_{n}}{u_{n}}=c$. Moreover, if we
assume that $\mathcal{H}_{w}\cong\mathcal{H}_{u}$ via an isomorphism
$T$ such that $T^{*}=\alpha C_{\phi}$, then the converse is also
true: $w_{n}\sim u_{n}$ and the isomorphism can be chosen to be isometric
if and only if there exists $c>0$ such that $\frac{w_{n}}{u_{n}}=c$.
\end{thm}

Note that some similar results for the multiplier algebras of a specific
family of weighted Hardy spaces are obtained in \cite[Section 7]{key-3-1}.

We now present a few more definitions and the rest of the questions
we tackled. Recall that the multiplier algebra of a Hilbert function
space $\mathcal{H}$ on $X$ is the set
\[
M(\mathcal{H})=\left\{ f\in\mathbb{C}^{X}:\forall h\in\mathcal{H},\ fh\in\mathcal{H}\right\} .
\]

The elements of $M(\mathcal{H})$ are called multipliers. Every multiplier
defines a bounded multiplication operator on $\mathcal{H}$. Under
the assumption that for every $x\in X$, there exists $h\in\mathcal{H}$
such that $h(x)\not=0$ (which will always hold in our case), $M(\mathcal{H})$
is a Banach algebra with respect to the operator norm \cite[Section 6.3.2]{key-1}.
\begin{defn}
Let $\mathcal{H}$ be a Hilbert function space $\mathcal{H}$ on $X$
. For every $A\subset X$ we define a subspace:
\[
\mathcal{H}_{A}=\overline{\operatorname{span}\left\{ k_{a}:a\in A\right\} }.
\]
This is clearly a Hilbert function space with respect to the restriction
of the inner product on $\mathcal{H}$. We denote the corresponding
multiplier algebra $\mathcal{M}_{A}=M(\mathcal{H}_{A})$.

The following example is the subject of our second classification
result. We denote the unit ball in $\mathbb{C}^{d}$ by $\mathbb{B}_{d}$.
\end{defn}

\begin{example}
Let $d$ be a positive integer and $t\in(0,\infty)$, then by \cite[p. 100, Remark 8.10]{bib:Shalit}
and \cite[p. 54]{key-3} there exists a Hilbert function space $\mathcal{H}_{d}^{t}$
on $\mathbb{B}_{d}$ with kernel given by:
\[
k^{\mathcal{H}_{d}^{t}}(x,y)=\frac{1}{(1-\langle x,y\rangle_{\mathbb{C}^{d}})^{t}}.
\]
The function $z\mapsto z^{t}$ is defined using the principal branch
of the logarithm on the half plane $P=\left\{ z\in\mathbb{C}:\mathfrak{R}(z)>0\right\} $.
Note that when $d$ is equal to 1, $\mathcal{H}_{d}^{t}$ is a weighted
Hardy space.
\end{example}

We can now present the rest of the problems that will be addressed
in this paper:

\subsection{Classification of $\mathcal{H}_{A}$ up to isometric isomorphism.
\protect \\
\protect \\
}

We define two subsets of the disk to be \emph{congruent }if there
exists a bi-holomorphic automorphism of $\mathbb{B}_{d}$ taking one
to the other.

We note that for $t=1$ the theorem follows from the results of \cite[Section 4]{key-7}
(see also the survey paper \cite{key-10}). Our proof, relying mostly
on linear algebra, is simpler and more direct. The partial solution
that will be presented in this paper is as follows:
\begin{thm}
\label{thm: 1.5}Let $\mathcal{H}=\mathcal{H}_{d}^{t}$ be as in Example
1.4. If $t\in(0,2]$ then for any two subsets $A,B\subset\mathbb{B}_{d}$,
$\mathcal{H}_{A}$ is isometric to $\mathcal{H}_{B}$ if and only
if $A$ and $B$ are congruent. If $t>2$ there exist non-congruent
subsets that yield isometric subspaces of $\mathcal{H}$.\\
\end{thm}

\subsection{Classification of $\mathcal{M}_{A}$ up to isometric isomorphism.\protect \\
\protect \\
}

We present a solution in the case where $\mathcal{H}$ is the classical
Hardy space on the disk. That is, the weighted Hardy space $\mathcal{H}_{w}$
with $w_{n}=1$ for all $n$. In the notation of Example 1.4, $\mathcal{H}=\mathcal{H}_{1}^{1}$.
We define for a subset $A\subset\mathbb{D}$:
\[
S(A):=\left\{ x\in\mathbb{D}:k_{x}\in\mathcal{H}_{A}\right\} .
\]
We will assume that $S(A)=A,\ S(B)=B$. This is quite a natural assumption
because clearly for any two subsets $A,A'\subset\mathbb{D}$:
\[
\mathcal{H}_{A}=\mathcal{H}_{A'}\iff S(A)=S(A').
\]
Moreover, the kernel functions $\{k_{x}^{\mathcal{H}}\}_{x\in\mathbb{D}}$
are easily seen to be linearly independent. Therefore if $A\subset\mathbb{D}$
is finite, then automatically $S(A)=A$.
\begin{thm}
\label{thm:1.6}Let $\mathcal{H}$ be the classical Hardy space on
the disk, then for any two subsets $A,B\subset\mathbb{D}$ such that
$S(A)=A$ and $S(B)=B$, $\mathcal{M}_{A}$ and $\mathcal{M}_{B}$
are isometrically isomorphic as Banach algebras if and only if $A$
and $B$ are congruent.
\end{thm}

Note that this result follows as a special case of \cite[Theorem 5.10]{key-7}.
Again, the method presented here is new and more direct. Also note
(although we will not use this fact) that if $A$ is an infinite proper
subset such that $S(A)=A$ , then by applying \cite[Proposition 2.2]{key-7}
we find that $A$ is the joint zero set of some bounded analytic functions
on the disk. By the main result of \cite[Theorem 2.1]{key-12} this
implies that $A$ is a Blaschke sequence. In general $S(A)=A$ if
and only if $A$ is a Blaschke sequence.

\subsection{Acknowledgements}

The results presented in this paper were obtained during the summer
projects program at the Technion Institute of Technology under the
guidance of Orr Moshe Shalit, Satish K. Pandey and Ran Kiri. We would
like to thank them for giving us the opportunity to learn and for
their attentive and careful instruction. We would like to thank the
referee for numerous helpful suggestions to clarify our ideas and
to make them more accurate.

\subsection{Data Availability}

Data sharing not applicable to this article as no datasets were generated
or analysed during our research.

\section{Proof of Theorem 1.2}

We remind the reader the statement of the theorem:

\subsubsection*{\textbf{\emph{\prettyref{thm:1.2}. }}If $w_{n}\sim u_{n}$, then
$\mathcal{H}_{w}\protect\cong\mathcal{H}_{u}$. The isomorphism can
be chosen to be isometric if and only if there exists $c>0$ such
that $\frac{w_{n}}{u_{n}}=c$. Moreover, if we assume that $\mathcal{H}_{w}\protect\cong\mathcal{H}_{u}$
via an isomorphism $T$ such that $T^{*}=\alpha C_{\phi}$ (where
$\phi:\mathbb{D}\rightarrow\mathbb{D}$ is a bijection), then the
converse is also true: $w_{n}\sim u_{n}$, and the isomorphism can
be chosen to be isometric if and only if there exists $c>0$ such
that $\frac{w_{n}}{u_{n}}=c$.}

To prove the theorem, we need the following facts:

1. If $\phi$ fixes the origin and $T^{*}=\alpha C_{\phi}$ then $T$
is diagonalized over the monomials.

2. If $T^{*}=\alpha C_{\phi}$ with $\phi$ which fixes the origin,
then $\mathcal{H}_{w}$ and $\mathcal{H}_{u}$ are isomorphic via
$T$ if and only if $w_{n}\sim u_{n}$. The isomorphism is isometric
if and only if there exists $c>0$ such that $\frac{w_{n}}{u_{n}}=c$.

3. If $\mathcal{H}_{w},\mathcal{H}_{u}$ are isomorphic via $T^{*}=\alpha C_{\phi}$,
then they are also isomorphic via $C_{\phi}$.

4. If $\mathcal{H}_{w},\mathcal{H}_{u}$ are isomorphic via $T^{*}=C_{\phi}$
then they are also isomorphic via some $\widetilde{T}^{*}=C_{\phi'}$
with $\phi'$ which fixes the origin.

Using these results, we can prove our main theorem.
\begin{proof}
If $w_{n}\sim u_{n}$, then we can construct such an isomorphism via
the construction described in the proof of $2$ (and the same construction
works for the isometric case). If $\mathcal{H}_{w}\cong\mathcal{H}_{u}$
via $T$ with $T^{*}=\alpha C_{\phi}$, then from $3$, we can without
loss of generality assume that $T^{*}=C_{\phi}$. Moreover, from $4$,
we might as well assume that $\phi$ fixes the origin, and so by $2$,
the result follows.
\end{proof}
Note that the lemmas we use rely heavily on $T^{*}$ being a scalar
multiple of a composition operator - the general case of a weighted
composition operator is more subtle, and the proofs we gave do not
hold for this case. Thus, we could only prove our classification result
under the given assumption. 

We now prodceed to prove the four facts stated above.
\begin{lem}
If $\phi$ fixes the origin and $T^{*}=\alpha C_{\phi}$ then $T$
is diagonalized over the monomials.
\end{lem}

\begin{proof}
First, note that if $C_{\phi}:\mathcal{H}_{w}\rightarrow\mathcal{H}_{u}$
is a composition operator, then $\phi$ is analytic. This is since
$C_{\phi}\left(id\right)=\phi$, which means that $\phi\in Im\left(C_{\phi}\right)\subset\mathcal{H}_{u}$.
But all functions in $\mathcal{H}_{u}$ are analytic, and so the result
follows. 

Now, if $\phi$ fixes the origin, then since $\phi$ is an analytic
bijection of the disk, then $\phi$ is a rotation - $\phi\left(z\right)=e^{i\theta}z$.
We thus get:
\[
T^{*}z^{n}=\alpha C_{\phi}z^{n}=\alpha e^{in\theta}z^{n}.
\]
Which means that $T^{*}$ is diagonalized over the monomials, and
so $T$ is as well.
\end{proof}
\begin{lem}
$\mathcal{H}_{w}$ and $\mathcal{H}_{u}$ are isomorphic via $T$
such that $T^{*}=\alpha C_{\phi}$ with $\phi$ which fixes the origin
if and only if $w_{n}\sim u_{n}$. The isomorphism is isometric if
and only if there exists $c>0$ such that $\frac{w_{n}}{u_{n}}=c$.
\end{lem}

\begin{proof}
$\Rightarrow$: Suppose that $\mathcal{H}_{w}\cong\mathcal{H}_{u}$
via $T$ such that $T^{*}=\alpha C_{\phi}$ with $\phi$ which fixes
the origin. Then by the previous lemma, we know that there is some
$\alpha\in\mathbb{C},\theta\in\mathbb{R}$ such that $T^{*}z^{n}=\alpha e^{in\theta}z^{n}$
for all $n\in\mathbb{N}$. Note that $T^{*}$ is by definition bounded
and invertible, and so there exist $\epsilon>0,M>0$ such that:
\[
\begin{aligned}\epsilon\left\Vert z^{n}\right\Vert _{\mathcal{H}_{u}} & \leq\left\Vert T^{*}z^{n}\right\Vert _{\mathcal{H}_{w}}\leq M\left\Vert z^{n}\right\Vert _{\mathcal{H}_{u}}\\
\Rightarrow\epsilon\sqrt{u_{n}} & \leq\left\Vert \alpha e^{in\theta}z^{n}\right\Vert _{\mathcal{H}_{w}}\leq M\sqrt{u_{n}}\\
\Rightarrow\epsilon\sqrt{u_{n}} & \leq\left|\alpha\right|\sqrt{w_{n}}\leq M\sqrt{u_{n}}\\
\Rightarrow\frac{\epsilon^{2}}{\left|\alpha\right|^{2}} & \leq\frac{w_{n}}{u_{n}}\leq\frac{M^{2}}{\left|\alpha\right|^{2}}.
\end{aligned}
\]
And so $w_{n}\sim u_{n}$. Moreover, if $T$ is isometric, then the
two sequences are proportional, since:
\[
\begin{array}{c}
\sqrt{u_{n}}=\left\Vert z^{n}\right\Vert _{\mathcal{H}_{u}}=\left\Vert T^{*}z^{n}\right\Vert _{\mathcal{H}_{w}}=\left|\alpha\right|\left\Vert z^{n}\right\Vert _{\mathcal{H}_{w}}=\left|\alpha\right|\sqrt{w_{n}}\\
\Rightarrow\frac{w_{n}}{u_{n}}=c.
\end{array}
\]

$\Leftarrow$: Given sequences $w_{n},\,u_{n}$ such that $w_{n}\sim u_{n}$,
we can choose $\alpha_{n}=\frac{w_{n}}{u_{n}}$ and construct the
linear map:
\[
T:\mathcal{H}_{w}\rightarrow\mathcal{H}_{u},\ Tz^{n}=\alpha_{n}z^{n}.
\]
$T$ is clearly invertible, and is in fact an RKHS isomorphism, since:
\[
T\left(k_{s}^{w}\right)=T\left(\sum_{n=0}^{\infty}\frac{\overline{s}^{n}}{w_{n}}z^{n}\right)=\sum_{n=0}^{\infty}\frac{\overline{s}^{n}}{w_{n}}T\left(z^{n}\right)
\]
\[
=\sum_{n=0}^{\infty}\frac{\overline{s}^{n}}{w_{n}}\cdot\frac{w_{n}}{u_{n}}z^{n}=\sum_{n=0}^{\infty}\frac{\overline{s}^{n}}{u_{n}}z^{n}=k_{s}^{u}.
\]
From here we also see that $T^{*}=C_{\phi}$ with $\phi=id$. Furthermore,
if $\frac{w_{n}}{u_{n}}=c>0$, then choose $\alpha_{n}=\sqrt{c}$.
$T$ is now an isometry:
\[
\left\Vert Tz_{n}\right\Vert _{\mathcal{H}_{u}}=\left\Vert \sqrt{c}z^{n}\right\Vert _{\mathcal{H}_{u}}=\sqrt{cu_{n}}=\sqrt{w_{n}}=\left\Vert z^{n}\right\Vert _{\mathcal{H}_{w}}.
\]
By a similar computation, one can conclude that $T$ is again an isomorphism,
such that $T^{*}=\frac{1}{\sqrt{c}}C_{\phi}$ with $\phi=id$.
\end{proof}
\begin{lem}
If $\mathcal{H}_{w},\mathcal{H}_{u}$ are isomorphic via $T^{*}=\alpha C_{\phi}$,
then they are also isomorphic via $C_{\phi}$
\end{lem}

\begin{proof}
Naturally, $C_{\phi}$ induces an RKHS isomorphism onto its image.
But its image is $\frac{1}{\alpha}\mathcal{H}_{u}=\mathcal{H}_{u}$. 
\end{proof}
The proof of the next lemma is an adaptation of the ``disk trick''
which was first put forward by Orr Shalit and Baruch Solel, and has
been applied to various operator algebraic settings in which there
is an action of the circle (see \cite{key-25} for more details).
\begin{lem}
If $T^{*}=C_{\phi}$ is an isomorphism, then there is an isomorphism
$\widetilde{T}$ such that $\widetilde{T}^{*}=C_{\phi'}$ and $\phi'$
fixes the origin
\end{lem}

\begin{proof}
Throughout this proof, in an effort to minimize verbiage, we will
call such composition operators isomorphisms as well.

Since $\mathcal{H}_{w}\cong\mathcal{H}_{u}$, we can define the following
non-empty set:
\[
O_{1}=\left\{ \lambda\in\mathbb{D}:\exists\psi\in\Aut\left(\mathbb{D}\right),\psi\left(\lambda\right)=0,\,C_{\psi}:\mathcal{H}_{w}\rightarrow\mathcal{H}_{u}\,\text{is an isomorphism}\right\} .
\]
If we show that $0\in O_{1}$, then we can conclude that there is
an isomorphism $C_{\phi}:\mathcal{H}_{w}\rightarrow\mathcal{H}_{u}$
with $\phi$ which fixes the origin, and so we are done.

Let $\lambda_{*}\in O_{1}$, which corresponds to some $C_{\phi}$,
where $\phi\in\Aut\left(\mathbb{D}\right)$. If $\lambda_{*}=0$,
we are done. Otherwise, we define the circle going through $\lambda_{*}$:
\[
C_{\lambda_{*}}=\left\{ e^{i\theta}\lambda_{*}:\theta\in\mathbb{R}\right\} .
\]
Denote $A=\phi\left(C_{\lambda_{*}}\right)\subset\mathbb{D}$ (see
Figure 2.1). Note that $A$ is a circle going through the origin.
It is a circle as the conformal image of a circle (since $\phi\in\Aut\left(\mathbb{D}\right)$),
and $0\in A$ since $\phi\left(\lambda_{*}\right)=0$ by definition.
We also define:
\[
O_{2}=\left\{ \lambda\in\mathbb{D}:\exists\tau\in\Aut\left(\mathbb{D}\right),\tau\left(\lambda\right)=0,\,C_{\tau}\in\Aut\left(\mathcal{H}_{w}\right)\right\} .
\]

Where $\Aut\left(\mathcal{H}_{w}\right)$ is the set of all isomorphism
from $\mathcal{H}_{w}$ to itself.

\textbf{Claim: $O_{1}$ and $O_{2}$ are invariant under rotations.
In particular, $C_{\lambda_{*}}\subset O_{1}$.}

Indeed, let $\lambda\in O_{1}$. We wish to show that $e^{i\theta}\lambda\in O_{1}$,
for all $\theta\in\mathbb{R}$. Clearly, $C_{e^{-i\theta}z}\in\Aut\left(\mathcal{H}_{w}\right)$.
Thus, if $\lambda$ corresponds to some $C_{f}$, then $e^{i\theta}\lambda$
will correspond to an isomorphism $C_{g}$ where $g\left(z\right)=f\left(e^{-i\theta}z\right)$
(so $C_{g}=C_{e^{-i\theta}z}C_{f}$), and so indeed $e^{i\theta}\lambda\in O_{1}$.
Similarly, one can show that $O_{2}$ is invariant under rotations.

\textbf{Claim: $A\subset O_{2}$.}

Let $\phi\left(\lambda'\right)\in A$, where $\lambda'\in C_{\lambda_{*}}$.
Since $\lambda'\in O_{1}$, we have some $\psi_{\lambda'}\in\Aut\left(\mathbb{D}\right)$
such that $\psi_{\lambda'}\left(\lambda'\right)=0$. Moreover, if
we define $\tau=\psi_{\lambda'}\circ\phi^{-1}$, we have that $C_{\tau}=C_{\phi^{-1}}C_{\psi_{\lambda'}}\in\Aut\left(\mathcal{H}_{w}\right)$
and:
\[
\tau\left(\phi\left(\lambda'\right)\right)=\psi_{\lambda'}\left(\phi^{-1}\left(\phi\left(\lambda'\right)\right)\right)=\psi_{\lambda'}\left(\lambda'\right)=0.
\]
And so $\phi\left(\lambda'\right)\in O_{2}$. We thus conclude that
$A\subset O_{2}$.

\begin{figure}
\centering{}\includegraphics[scale=0.4]{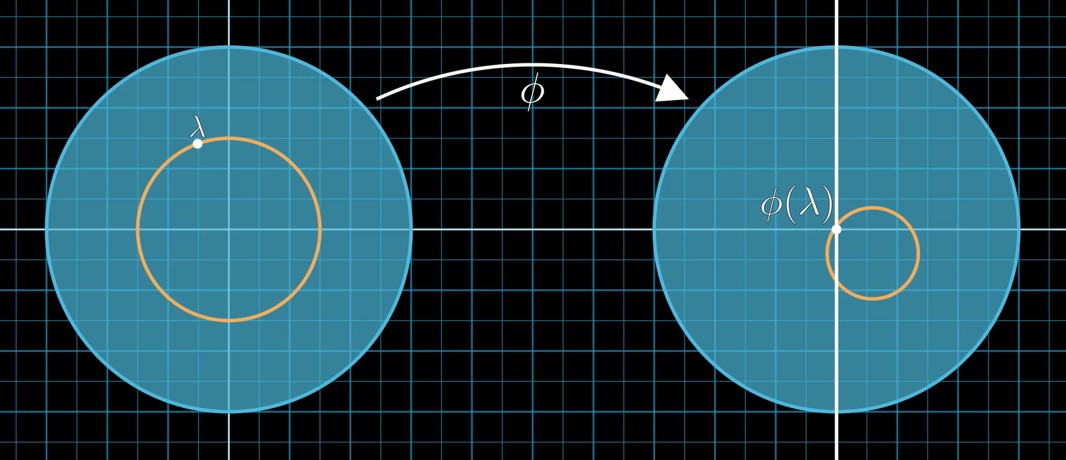}\caption{$C_{\lambda_{*}}$ is sent by $\phi$ to another circle, passing through
the origin}
\end{figure}

Define $\left[A\right]$ to be the open disk enclosed by $A$. Since
$O_{2}$ is invariant under rotations, we can rotate $A$ about the
origin and obtain a large disk around the origin (see Figure 2.2),
which contains $\left[A\right]$, which means that $\left[A\right]\subset O_{2}$.
Note that $\phi^{-1}\left(\left[A\right]\right)$ is exactly the open
disk enclosed by $C_{\lambda_{*}}$. This is since $\phi^{-1}$ must
send $\left[A\right]$ to the interior or the exterior of $C_{\lambda_{*}}$
(which is a Jordan curve) in $\mathbb{C}$, but $\phi^{-1}\left(\left[A\right]\right)$
must also be bounded in the disk and so it must be the interior of
the disk enclosed by $C_{\lambda_{*}}$.

\begin{figure}
\centering{}\includegraphics[scale=0.4]{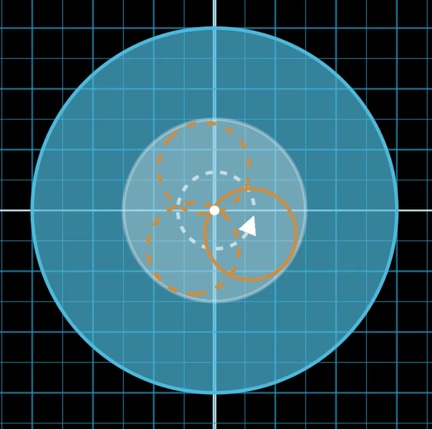}\caption{The set $A$ can be rotated about the origin, to obtain a disk}
\end{figure}

Since $\phi^{-1}\left(\left[A\right]\right)$ is a disk centered around
the origin, we conclude that $0\in\phi^{-1}\left(\left[A\right]\right)$.
But this means that $0\in O_{1}$ as well. Indeed, Since $0\in\left[C_{\lambda_{*}}\right]=\phi^{-1}\left(\left[A\right]\right)$,
we can write $0=\phi^{-1}\left(a\right)$, where $a\in\left[A\right]\subset O_{2}$.
By definition of $O_{2}$, there is some $\tau\in\Aut\left(\mathbb{D}\right)$
such that $\tau\left(a\right)=0$ and $C_{\tau}\in\Aut\left(\mathcal{H}_{w}\right)$.
Then if we look at $\psi\in\Aut\left(\mathbb{D}\right),\,\psi=\tau\circ\phi$,
then $C_{\psi}=C_{\phi}C_{\tau}:\mathcal{H}_{w}\rightarrow\mathcal{H}_{u}$
is an isomorphism, and moreover:
\[
\psi\left(0\right)=\tau\left(\phi\left(\phi^{-1}\left(a\right)\right)\right)=\tau\left(a\right)=0.
\]

We can thus conclude that $0\in O_{1}$, and so by the discussion
above this completes the proof.
\end{proof}
Note that lemma $2.4$ can be proven using a different approach. By
corollary $9.10$ of \cite{M17}, one can show that if $T^{*}=C_{\phi}$
with $\phi$ which does not fix the origin, then $\mathcal{H}_{2}$
is invariant under $\Aut\left(\mathbb{D}\right)$. This gives that
$C_{z}$ is an isomorphism between $\mathcal{H}_{1}$ and $\mathcal{H}_{2}$
as well. Since corollary $9.10$ of \cite{M17} holds for spaces on
the ball as well, this approach has the benefit of giving a possible
generalization of our result to the ball as well, although we do not
pursue this idea here.

\section{Proofs of theorems 1.5 and 1.6}

We first present a useful characterization of Hilbert function space
isometries. Let $\mathcal{H}$ and $\mathcal{E}$ be Hilbert function
spaces on sets $X$ and $Y$ respectively. We say that a function
$\phi:X\to Y$ induces a morphism of Hilbert function spaces $T:\mathcal{H}\to\mathcal{E}$
if $T(k_{z}^{\mathcal{H}})$ is proportional to $k_{\phi(z)}^{\mathcal{E}}$
for all $z\in X$. 
\begin{lem}
Let $\mathcal{H}$ and $\mathcal{E}$ be Hilbert function spaces on
$X$ and $Y$ respectively, and assume that $z\mapsto k_{z}^{\mathcal{H}}(z)$
is nowhere vanishing. Then for any function $\phi:X\to Y$ the following
are equivalent:$\phantom{text}$
\begin{enumerate}
\item The function $\phi$ induces an isometric isomorphism between $\mathcal{H}$
and $\mathcal{E}$.
\item There exists a function $f:X\to\mathbb{C}$ such that for all $z,w\in X$:
\[
k_{z}^{\mathcal{H}}(w)=f(z)\overline{f(w)}k_{\phi(z)}^{\mathcal{E}}(\phi(w)).
\]
\end{enumerate}
\end{lem}

\begin{proof}
$(1)\implies(2):$ If $k_{x}^{\mathcal{H}}\mapsto f(x)k_{\phi(x)}^{\mathcal{E}}$
extends to an isometry for some $f:X\to\mathbb{C}$, then for any
$z,w\in X:$
\[
\begin{aligned}k_{z}^{\mathcal{H}}(w)= & \langle k_{z}^{\mathcal{H}},k_{w}^{\mathcal{H}}\rangle\\
= & \langle f(w)k_{\phi(z)}^{\mathcal{E}},f(w)k_{\phi(w)}^{\mathcal{E}}\rangle\\
= & f(z)\overline{f(w)}k_{\phi(z)}^{\mathcal{E}}(\phi(w))
\end{aligned}
\]
as required.

$(2)\implies(1):$ Assume that $k_{z}^{\mathcal{H}}(w)=f(z)\overline{f(w)}k_{\phi(z)}^{\mathcal{E}}(\phi(w))$
holds for all $z,w\in X$. Then for all $z,w\in X$:
\[
\begin{aligned}\langle k_{z}^{\mathcal{H}},k_{w}^{\mathcal{H}}\rangle= & k_{z}^{\mathcal{H}}(w)\\
= & f(z)\overline{f(w)}k_{\phi(z)}^{\mathcal{E}}(\phi(w))\\
= & \langle f(w)k_{\phi(z)}^{\mathcal{E}},f(w)k_{\phi(w)}^{\mathcal{E}}\rangle
\end{aligned}
\]
Notice that because $z\mapsto k_{z}^{\mathcal{H}}(z)$ is nowhere
vanishing, so is $f$. Because any function which is orthogonal to
all kernel functions is everywhere zero, $\{k_{x}^{\mathcal{H}}\}_{x\in X}$
and $\{k_{y}^{\mathcal{E}}\}_{y\in Y}$ are dense in $\mathcal{H}$
and $\mathcal{E}$ respectively. Therefore $k_{x}^{\mathcal{H}}\mapsto f(x)k_{\phi(x)}^{\mathcal{E}}$
extends to an isometric isomorphism between $\mathcal{H}$ and $\mathcal{E}$
(note that we need $f$ to be nowhere vanishing to guarantee surjectivity). 
\end{proof}
\begin{rem}
If $X=\{x_{1},\dots,x_{n}\}$ then it will be convenient to state
the above condition in matrix form as:
\[
(a):\ \ \left[k^{\mathcal{H}}(x_{i},x_{j})\right]=\left[f(x_{i})\overline{f(x_{j})}k^{\mathcal{E}}(\phi(x_{i}),\phi(x_{j}))\right]
\]
Or when $k^{\mathcal{E}}(\phi(x_{i}),\phi(x_{j}))$ does not vanish
as:
\[
(b):\ \ \left[\frac{k^{\mathcal{H}}(x_{i},x_{j})}{k^{\mathcal{E}}(\phi(x_{i}),\phi(x_{j}))}\right]=\left[f(x_{i})\overline{f(x_{j})}\right]
\]
\end{rem}

\subsection{Moving towards a proof of Theorem 1.5.\protect \\
}

Throughout this section fix $d\in\mathbb{N}$, $t\in(0,\infty)$,
and let $\mathcal{H}=\mathcal{H}_{d}^{t}$ be the Hilbert function
spaces on $\mathbb{B}_{d}$ defined as in Example 1.4. We denote the
kernel function of $\mathcal{H}$ by $k$.

We denote the bi-holomorphic automorphisms of the ball by $\Aut(\mathbb{B}_{d})$.
From now on we omit the adjective ``bi-holomorphic''; by\emph{ automorphism}
we shall mean bi-holomorphic automorphism. The following Proposition
is well known. We include a proof for completeness.
\begin{prop}
For any $\phi\in\Aut(\mathbb{B}_{d})$ there exists $f:\mathbb{B}_{d}\rightarrow\mathbb{C}$
such that $k_{x}\mapsto\overline{f(x)}k_{\phi(x)}$ extends to an
isometric automorphism of $\mathcal{H}$.
\end{prop}

\begin{proof}
Let $\phi\in\Aut(\mathbb{B}_{d})$ and assume that $\phi(a)=0$ then
by \cite[Theorem 2.2.2]{key-8},we have for all $x,y\in\mathbb{B}_{d}$:
\[
1-\langle\phi(x),\phi(y)\rangle=(1-|a|^{2})\frac{1-\langle x,y\rangle}{(1-\langle x,a\rangle)(1-\langle a,y\rangle)}.
\]
This implies that for $f:x\mapsto\frac{\sqrt{1-|a|^{2}}}{1-\langle x,a\rangle}$
the following equality holds for all $x,y\in\mathbb{B}_{d}$:
\[
(c):\ \ k(\phi(x),\phi(y))f^{t}(x)\overline{f^{t}(y)}=k(x,y).
\]
Indeed, using the principal branch $z^{t}=e^{t\Log(z)}$, the identity
$z^{t}w^{t}=(zw)^{t}$ holds for any two complex numbers in the half
plane $P=\left\{ z\in\mathbb{C}:\mathfrak{R}(z)>0\right\} $. This
implies $(c)$ for all $x,y\in\mathbb{B}_{d}$ sufficiently close
to the origin. Then $(c)$ follows for all $x,y\in\mathbb{B}_{d}$
by analyticity of both sides of the equation. By Lemma 3.1 this shows
that $k_{x}\mapsto\overline{f(x)^{t}}k_{\phi(x)}$ extends to an isometric
automorphism of $\mathcal{H}$.
\end{proof}
\begin{cor}
If $A,B\subset\mathbb{B}_{d}$ are congruent, then $\mathcal{H}_{A}$
is isometrically isomorphic to $\mathcal{H}_{B}$.
\end{cor}

\begin{proof}
This follows from Proposition 3.1 and Proposition 3.3 because the
restriction of an isometric isomorphism is an isometric isomorphism.
\end{proof}
We will call $\mathcal{H}$ \emph{faithful} if for any two subsets
$A,B\subset\mathbb{B}_{d}$, every $\phi:A\to B$ which induces an
isometric isomorphism $\mathcal{H}_{A}\to\mathcal{H}_{B}$ may be
extended to an automorphism of $\mathbb{B}_{d}$.

The last ingredient for the proof of Theorem 1.5 is:
\begin{prop}
$\mathcal{H}$ is faithful if and only if $t\in(0,2]$.
\end{prop}

For the proof we will need the following two lemmas:
\begin{lem}
Let $A,B\subset\mathbb{C}^{d}$ , and assume that $f:A\to B$ preserves
the inner product, meaning that for any $a,a'\in A$:
\[
\langle a,a'\rangle=\langle f(a),f(a')\rangle.
\]
Then there exists a unitary operator $U:\mathbb{C}^{d}\to\mathbb{C}^{d}$
extending $f$.
\end{lem}

\begin{proof}
By basic linear algebra, there exists a unitary operator $U':\span A\to\span B$.
Choose orthonormal bases $\{e_{i}\}_{i=1}^{r},\{e'_{i}\}_{i=1}^{r}$
for $\left(\span A\right)^{\perp}$ and $\left(\span B\right)^{\perp}$
respectively, and let $T$ be the partial isometry defined by linearly
extending $e_{i}\mapsto e'_{i}$, it is easy to check that 
\[
U=U'\oplus T:\span A\oplus\left(\span A\right)^{\perp}\to\span B\oplus\left(\span B\right)^{\perp}
\]
is the required isometry.
\end{proof}
\begin{lem}
Let $t\in(0,\infty)$ and $U=\mathbb{D}\setminus\{0\}$. The function
$g_{t}:U\to\mathbb{C},\ z\mapsto(1-z)^{t}$ defined by the principal
branch of the logarithm is injective if $t\leq2$. If $t>2$ there
exists $z,w\in U$ such that $z\neq w$ and $g_{t}(z)=g_{t}(w)$ and
$|z|=|w|$.
\end{lem}

\begin{proof}
Assume first that $t\in(0,2]$ and $g_{t}(z)=g_{t}(w)$ for $z,w\in U$.
Let $z'=1-z$ and $w'=1-w$. Since $e^{t\Log z'}=e^{t\Log w'}$ there
exists an integer $m$ such that $t(\Arg(z')-\Arg(w'))=2\pi m$. Since
$\mathfrak{R}(z'),\mathfrak{R}(w')>0$ we have $\Arg(z'),\Arg(w')\in(-\frac{\pi}{2},\frac{\pi}{2})$
and therefore:
\[
\pi|m|\leq\frac{2\pi|m|}{t}=\Arg(z')-\Arg(w')\in(-\pi,\pi).
\]
Therefore we must have $m=0$ which implies $z=w$. If $t>2$, then
we can find $z\in U$ such that $t\Arg(z')=\pi$ (in particular $z\neq\overline{z})$.
For $w=\overline{z}$ we get $t(\Arg(z')-\Arg(w'))=2\pi$ and therefore
$g_{t}(z)=g_{t}(w)$. As we wanted to show.
\end{proof}

\subsection*{Proof of Proposition 3.5:}
\begin{proof}
We first assume that $t\in(0,2]$ and show that $\mathcal{H}$ is
faithful.

Let $A,B$ be subsets of $\mathbb{B}_{d}$ and assume that $T:\mathcal{H}_{A}\to\mathcal{H}_{B}$
is an isometry induced by $\phi:A\to B$. We deal first with the case
where $0\in A\cap B$ and $\phi(0)=0$.\\

Let $a_{2},a_{3}\in A\setminus\{0\}$, and denote $a_{1}=0$. By Lemma
3.1 there exists $\delta:\{a_{1},a_{2},a_{3}\}\to\mathbb{C}$ such
that:
\[
\left[\delta(a_{i})\overline{\delta(a_{j})}\right]=\left[\frac{k(a_{i},a_{j})}{k(\phi(a_{i}),\phi(a_{j}))}\right]=\left[\frac{\left(1-\langle\phi(a_{i}),\phi(a_{j})\rangle\right)^{t}}{\left(1-\langle a_{i},a_{j}\rangle\right)^{t}}\right]
\]
Notice the identity $(z^{-1})^{t}=(z^{t})^{-1}$ holds for $\mathfrak{R}(z)>0$
because $\Log(z^{-1})=-\Log(z)$. This gives us:
\[
\Rank\left[\frac{\left(1-\langle\phi(a_{i}),\phi(a_{j})\rangle\right)^{t}}{\left(1-\langle a_{i},a_{j}\rangle\right)^{t}}\right]=1.
\]
Which implies that the $2\times2$ minors of the matrix vanish. 

Because $a_{1}=\phi(a_{1})=0$, if $i=1$ or $j=1$ we have:
\[
k(a_{i},a_{j})=k(\phi(a_{i}),\phi(a_{j}))=1.
\]
Therefore the following $2\times2$ minor vanishes:
\[
0=\det\begin{bmatrix}1 & 1\\
1 & \frac{\left(1-\langle\phi(a_{i}),\phi(a_{j})\rangle\right)^{t}}{\left(1-\langle a_{i},a_{j}\rangle\right)^{t}}
\end{bmatrix}=\frac{\left(1-\langle\phi(a_{i}),\phi(a_{j})\rangle\right)^{t}}{\left(1-\langle a_{i},a_{j}\rangle\right)^{t}}-1.
\]
Therefore $(1-\langle\phi(a_{2}),\phi(a_{3})\rangle)^{t}=(1-\langle a_{2},a_{3}\rangle)^{t}$.
By Lemma 3.7 this implies $\langle\phi(a_{2}),\phi(a_{3})\rangle=\langle a_{2},a_{3}\rangle$.
Since $a_{2},a_{3}\in A\setminus\{0\}$ were arbitrary this shows
that $\phi$ preserves the inner product and therefore it may be extended
to a linear isometry by Lemma 3.5.

For the general case, we use the transitivity of $\Aut(\mathbb{B}_{d})$
(see \cite[Theorem 2.2.3]{key-8}). Choose some $a\in A$, by transitivity
there exist automorphisms $\psi,\theta\in\Aut(\mathbb{B}_{d})$ such
that:
\[
\psi(0)=a,\ \ \theta(\phi(a))=0.
\]
By Proposition 3.3, $\psi$ and $\theta$ induce isometric isomorphisms
and therefore 
\[
g:=\theta\circ\phi\circ\psi:\psi^{-1}(A)\to\theta(B)
\]
induces an isometric isomorphism between $\mathcal{H}_{\psi^{-1}(A)}$
and $\mathcal{H}_{\theta(B)}$ and fixes the origin. By the first
case it extends to an automorphism $\tilde{g}\in\Aut(\mathbb{B}_{d})$.
Then $\theta^{-1}\circ\tilde{g}\circ\psi^{-1}$ is an automorphism
of the ball which extends $\phi$. This shows that if $t\in(0,2]$
then $\mathcal{H}$ is faithful.\\

Assume that $t>2$. Then by Lemma 3.7 there exists $z,w\in\mathbb{D}\setminus\{0\}$
such that $(1-z)^{t}=(1-w)^{t}$, $|z|=|w|$ and $z\neq w$. It is
easy to see that we can choose $a_{1},a_{2},b_{1},b_{2}$ in $\mathbb{B}_{d}$
such that $||a_{i}||=||b_{i}||$ and 
\[
\langle a_{1},a_{2}\rangle=z,\ \langle b_{1},b_{2}\rangle=w.
\]
Indeed, one could simply take: 
\[
a_{1}=\frac{z}{\sqrt{|z|}}e_{1},\ a_{2}=\sqrt{|z|}e_{1},\ b_{1}=\frac{w}{\sqrt{|w|}}e_{1},\ b_{2}=\sqrt{|w|}e_{1}
\]
where $e_{1}$ is the vector $(1,0,\dots,0)\in\mathbb{C}^{d}$. Put
$A=\{0,a_{1},a_{2}\},\ B=\{0,b_{1},b_{2}\}$ and let $\phi:A\to B$
be given by $\phi(0)=0$, $\phi(a_{i})=b_{i}$. Then we get the following
equality of matrices:
\[
\left[\frac{k(a_{i},a_{j})}{k(\phi(a_{i}),\phi(a_{j}))}\right]_{i,j=1,2,3}=\begin{bmatrix}1 & 1 & 1\\
1 & 1 & 1\\
1 & 1 & 1
\end{bmatrix}
\]
Therefore ,by Lemma $3.1$, $\phi$ induces an isometry between $\mathcal{H}_{A}$
and $\mathcal{H}_{B}$. On the other hand, any automorphism of the
ball fixing the origin must be a unitary (see \cite[Theorem 2.2.5]{key-8}),
so $\phi$ cannot be extended to an automorphism. This shows that
if $t>2$, then $\mathcal{H}$ is not faithful.
\end{proof}
We can now restate the theorem and prove it:

\subsubsection*{\textbf{\emph{\prettyref{thm: 1.5}. }}Let $\mathcal{H}=\mathcal{H}^{t,d}$
be as in Example 1.4. If $t\in(0,2]$ then for any two subsets $A,B\subset\mathbb{B}_{d}$,
$\mathcal{H}_{A}$ is isometric to $\mathcal{H}_{B}$ if and only
if $A$ and $B$ are congruent. If $t>2$ there exist non-congruent
subsets that yield isometric subspaces of $\mathcal{H}$.}
\begin{proof}
We have already costructed examples of non-congruent subsets yielding
isometric subsets of $\mathcal{H}$ for $d\in\mathbb{N},\ t>2$ in
the proof of Proposition 3.5.

Assume that $t\in(0,2],d\in\mathbb{N}$ and let $A,B\subset\mathbb{B}_{d}$.
Then $\mathcal{H}$ is faithful by Proposition 3.5. Therefore if $T:\mathcal{H}_{A}\to\mathcal{H}_{B}$
is an isometry induced by a bijection $\phi:A\to B$, then $\phi$
extends to an automorphism of the ball which shows that $A$ and $B$
are congruent. Conversely if $A$ is congruent to $B$ then $\mathcal{H}_{A}$
and $\mathcal{H}_{B}$ are isometric by Corollary 3.4.
\end{proof}

\subsection{Proof of Theorem 1.6.\protect \\
}

Throughout this section let $\mathcal{H}=\mathcal{H}_{1}^{1}$ be
the Hilbert function space on $\mathbb{D}$ defined as in Example
1.4 (this space is known as the classical Hardy space and is often
denoted $H^{2}$). 

Our goal is then to prove that for any 2 subsets $A,B\subset\mathbb{D}$,
with $S(A)=A,\ S(B)=B$, $\mathcal{M}_{A}$ and $\mathcal{M}_{B}$
are isometric if and only if $A$ and $B$ are congruent. \\

If $\phi$ is the restriction of an automorphism which takes $B$
to $A$, then by Corollary 3.5 it induces an isometry between $\mathcal{H}_{B}$
and $\mathcal{H}_{A}$. It is well known that this implies that the
pull back: 
\[
\phi^{*}:\mathcal{M}_{A}\to\mathcal{\mathcal{M}}_{B},\ f\mapsto f\circ\phi.
\]
Is a well-defined isometric isomorphism, see for example \cite[Section 2.6, p. 25]{bib:Shalit}. 
\begin{thm}
For any $A\subset\mathbb{D}$ and $f\in\mathcal{M}_{A}$ , there exists
an extension of $f$ to the unit disk $\tilde{f}\in\mathcal{M}_{\mathbb{D}}$
such that :
\[
||\tilde{f}||_{\mathcal{M}_{\mathbb{D}}}=||f||_{\mathcal{M}_{A}}.
\]
\end{thm}

\begin{proof}
The proof is an adaptation of the well know proof of Pick's theorem
via commutant lifting (see \cite[p. 163]{bib:Shalit}) and is included
for completeness. 

Let $0\not=f\in\mathcal{M}_{A}$. We wish to extend $f$ to a multiplier
on the disk of norm $||f||_{\mathcal{M}_{A}}$, by normalizing we
may assume that $||f||_{\mathcal{M}_{A}}=1$. Let $S\in B(\mathcal{H})$
be the operator corresponding to multiplication by $z$. Notice that
for all $\lambda\in\mathbb{D},$ $k_{\lambda}$ is an eigenvector
of $S^{*}$ (see, e.g., \cite[Proposition 6.3.5.]{key-1}). This implies
that $\mathcal{H}_{A}$ is $S^{*}\text{-invariant}$. Let $T\in B(\mathcal{H}_{A})$
be the operator corresponding to multiplication by $f.$ Then $T^{*}$
commutes with the restriction of $S^{*}$ to $\mathcal{H}_{A}$, because
their adjoints commute. By a corollary to the commutant lifting theorem
(see \cite[Corollary 10.30]{bib:Shalit}), this implies that $f$
extends to a multiplier $\tilde{f}\in\mathcal{M}_{\mathbb{D}}$ of
norm at most one. The inequality $||\tilde{f}||_{\mathcal{M}_{\mathbb{D}}}\geq||T^{*}||=1$
holds because $T^{*}$ is the restriction of the adjoint of multiplication
by $\tilde{f}$ to $\mathcal{H}_{A}$.
\end{proof}
\begin{defn*}
The \emph{pseudo-hyperbolic metric} on the disk $\rho$ is a metric
on $\mathbb{D}$ given by the formula:
\[
\rho(x,y)=\left|\frac{x-y}{1-\overline{x}y}\right|.
\]
\end{defn*}
\begin{thm*}
(Schwarz-Pick Lemma, see \cite[Theorem 4, p. 15]{key-2}) Let $f:\mathbb{D}\to\mathbb{D}$
be holomorphic. Then for any $x,y\in\mathbb{D}$:
\[
\rho(f(x),f(y))\leq\rho(x,y).
\]
If equality is achieved for some pair of points, then $f$ is a conformal
automorphism of the disk.
\end{thm*}
It is well known that $\mathcal{M}_{\mathbb{D}}$ is equal to the
algebra of bounded analytic functions on the disk equipped with the
supremum norm (see, e.g., \cite[Example 6.3.9]{key-1}). This algebra
is usually denoted by $H^{\infty}$. In this paper we keep to the
notation $\mathcal{M}_{\mathbb{D}}$ as it emphasizes the role $H^{\infty}$
plays in the proof. The following lemma is well known. We include
the proof for completeness.
\begin{lem}
Let $\psi:\mathcal{M}_{\mathbb{D}}\to\mathbb{C}$ be a unital morphism
of algebras such that $\psi(z)=a\in\mathbb{D}$. Then $\psi$ is given
by evaluation at $a$.
\end{lem}

\begin{proof}
Let $f\in\mathcal{M}_{\mathbb{D}}$. By Taylor's Theorem we can find
an analytic function $g:\mathbb{D}\to\mathbb{C}$ such that:
\[
\frac{f(z)-f(a)}{z-a}=g(z).
\]
It follows that $g$ is bounded on the disk, which implies $g\in\mathcal{M}_{\mathbb{D}}$.
We calculate:
\[
\begin{aligned}\psi(f)= & \psi(f(a)+(z-a)g(z))\\
= & f(a)+(\psi(z)-a)\psi(g(z))\\
= & f(a)
\end{aligned}
\]
As we wanted to show.
\end{proof}
We say that a morphism of algebras $\Phi:\mathcal{M}_{A}\to\mathcal{M}_{B}$
is a \emph{pull-back} by $\phi:B\to A$ if $\Phi(f)=f\circ\phi$ for
all $f\in\mathcal{M}_{A}$. We denote such a pull-back by $\Phi=\phi^{*}$.
\begin{prop}
If $\Phi:\mathcal{M}_{A}\to\mathcal{M}_{B}$ is an isometric isomorphism
of Banach algebras, where $A=S(A)$ and $B=S(B)$, then it is a pull-back
by $\phi:B\to A$ where $\phi$ is the restriction of an automorphism
of the disk.
\end{prop}

We will prove this in two stages:
\begin{claim*}
There exists $\phi:B\to A$ for which $\Phi=\phi^{*}$.
\end{claim*}
\begin{proof}
Let $\phi=\Phi(z)$, i.e. $\phi$ is the image of the identity function
under $\Phi$. We will show that $\Phi=\phi^{*}$.

It is easy to see that $||z||_{\mathcal{M}_{A}}\leq1$ (by \cite[Proposition 6.3.15]{key-1},
for example) and therefore we get $||\phi||_{\mathcal{M}_{B}}\leq1$.
By Theorem 3.8 we get an extension $\tilde{\phi}\in\mathcal{M}_{\mathbb{D}}$
of $\phi$, such that $||\tilde{\phi}||_{\infty}\leq1$. Notice that
$\phi$ cannot be constant because $z$ is not constant and $\Phi^{-1}$
fixes the constant functions (it is an isomorphism of algebras). By
the maximum modulus principle we get that $\tilde{\phi}(\mathbb{D})\subset\mathbb{D}$.
\\

Let $b\in B$ and define the following homomorphism of algebras: 
\[
\psi:\mathcal{M}_{\mathbb{D}}\to\mathbb{C},\ \ f\mapsto\Phi(f_{|A})(b).
\]
Here $f_{|A}$ is the restriction of $f$ to $A$ (this is a multiplier
by \cite[Proposition 6.3.5.]{key-1}). Because $\psi(z)=\phi(b)\in\mathbb{D}$,
by Lemma 3.9 $\psi$ is given by evaluation at $\phi(b)$. We show
that $\phi(b)\in A$ by adapting \cite[Theorem 9.27]{bib:Shalit}.

We assume that $\phi(b)\in\mathbb{D}\setminus A$ and reach a contradiction.
This implies $k_{\phi(b)}\not\in\mathcal{H}_{A}$ because $S(A)=A$.
Therefore we can define a bounded operator $T$ on $\mathcal{H}_{A}\oplus\span\{k_{\phi(b)}\}$
sending $\mathcal{H}_{A}$ to zero and fixing $k_{\phi(b)}$. The
adjoint of $T$ is a multiplier by \cite[Proposition 6.3.5]{key-1},
which extends to a multiplier $g\in\mathcal{M}_{\mathbb{D}}$ by Theorem
3.8. Now we notice that $g(a)=0$ for all $a\in A$ by definition.
This implies:
\[
1=g(\phi(b))=\psi(g)=\Phi(g_{|A})(b)=\Phi(0)(b)=0.
\]
We reached a contradiction, so it cannot be that $\phi(b)\not\in A$.
This shows that the image of $\phi$ is contained in $A$. By Theorem
3.8 for any $f\in\mathcal{M}_{A}$ there exists an extension $\tilde{f}$
of $f$ to $\mathcal{M}_{\mathbb{D}}$. We calculate:
\[
\Phi(f)(b)=\Phi(\tilde{f}_{|A})(b)=\psi(\tilde{f})=f(\phi(b))
\]
Therefore $\Phi=\phi^{*}$, as we wanted to show.
\end{proof}
\begin{claim*}
The extension $\tilde{\phi}$ is a conformal automorphism of the disk.
\end{claim*}
\begin{proof}
Because $\tilde{\phi}:\mathbb{D}\to\mathbb{D}$ is analytic, the Schwarz-Pick
Lemma implies that for any $x,y\in B$:

\[
\rho(\phi(x),\phi(y))=\rho(\tilde{\phi(x)},\tilde{\phi(y)})\leq\rho(x,y).
\]
But $(\phi^{-1})^{*}=\Phi^{-1}$ is also an isometric isomorphism,
so by reversing the argument we get that for any $x,y\in B$:
\[
\rho(x,y)=\rho(\phi^{-1}(\phi(x)),\phi^{-1}(\phi(y)))\leq\rho(\phi(x),\phi(y)).
\]
We deduce that $\tilde{\phi}$ must be an automorphism by the Schwarz-Pick
lemma because it preserves the pseudo-hyperbolic metric on $B$. This
finishes the proof of the proposition, and therefore proves Theorem
1.6.
\end{proof}


\begin{thebibliography}{10}
\bibitem{bib:Shalit} J. Agler, J. E. McCarthy; \emph{Pick interpolation
and Hilbert Function Spaces}; Graduate Studies in Mathematics Vol
44; American Mathematical Society; Rhode Island; 2000.

\bibitem{key-22}D. Alpay, M. Putinar and V. Vinnikov;\emph{ A Hilbert
space approach to bounded analytic extension in the ball}; Commun.
Pure Appl. Anal. 2 (2003); 139-{}-145.

\bibitem{key-222}N. Arcozzi, R. Rochberg and E.T. Sawyer; \emph{Carleson
measures for the Drury-Arveson Hardy space and other Besov-Sobolev
spaces on complex balls}; Adv. Math. 218 (2008), 1107-{}-1180.

\bibitem{key-4} C. Cowen, B. MacCluer; \emph{Composition Operators
on Spaces of Analytic Functions}, Studies in Advanced Mathematics;
CRC Press; Boca Raton; 1995.

\bibitem{key-3-1} K.R. Davidson, M. Hartz, O.M. Shalit\emph{; Multipliers
of Embedded Disks}; Complex Anal. Oper. Theory 9 (2015), 287-{}-321. 

\bibitem{key-97}K.R. Davidson, C. Ramsey and O.M. Shalit; \emph{The
isomorphism problem for some universal operator algebras}; Adv. Math.
228 (2011), 167-{}-218. 

\bibitem{key-7} K.R. Davidison, C. Ramsey, O.M. Shalit; \emph{Operator
Algebras For Analytic Varieties}; Transactions of the American Mathematical
Society 367 (2015), 1121--1150. 

\bibitem{key-12}J.B. Garnett; \emph{Bounded Analytic Functions};
Springer; San Diego; 2007;

\bibitem{M12}M. Hartz; \emph{Topological isomorphisms for some universal
operator algebras}; J. Funct. Anal., 263 (2012), 3564-{}-3587. 

\bibitem{M17}M. Hartz; \emph{On the isomorphism problem for multiplier
algebras of Nevanlinna-Pick spaces}; Canad. J. Math. 69 (2017), 54-{}-106. 

\bibitem{key-19}M. Kerr, J.E. McCarthy and O.M. Shalit; \emph{On
the isomorphism question for complete Pick multiplier algebras}; Integral
Equations Operator Theory 76 (2013), 39-{}-53.

\bibitem{key-2}S. Krantz; \emph{Complex Analysis: The Geometric Viewpoint};
Carus Mathematical Monographs 23; Mathematical Association of America;
Washington; 1990.

\bibitem{key-9}S. Krantz; \emph{Geometric Function Theory}; Cornerstones;
Birkhauser; Boston; 2006.

\bibitem{key-3}V. I. Paulsen; \emph{An Introduction to the Theory
of Reproducing Kernel Hilbert Spaces}; Cambridge Studies in Advanced
Mathematics; Cambridge University Press; Cambridge; 2016. 

\bibitem{key-6}R. Rochberg;\emph{ Complex Hyperbolic Geometry and
Hilbert Spaces with Complete Pick Kernels}; Journal of Functional
Analysis 276 (2018); 1622--1679.

\bibitem{key-10}G. Salomon, O.M. Shalit; \emph{The Isomorphism Problem
For Complete Pick Algebras: a Survey}; Operator Theory, Function Spaces,
and Applications. Operator Theory: Advances and Applications 255 (2016);
Birkhäuser, Cham; 167--198;

\bibitem{key-1}O.M. Shalit; \emph{A First Course in Functional Analysis};
CRC Press; Boca Raton; 2017. 

\bibitem{key-25}O.M. Shalit; \emph{The disc trick (and some other
cute moves); }Posted on Shalit's blog: ``Noncommutative Analysis''
(2019); https://noncommutativeanalysis.com/2019/10/10/the-disc-trick-and-some-other-cute-moves/. 

\bibitem{key-8}R. Walter; \emph{Function Theory in The Unit Ball};
Classics in Mathematics; Springer-Verlag; Berlin; 2008.
\end{thebibliography}
\end{document}